\documentclass[a4paper,11pt,reqno,intlimits]{extarticle}
\usepackage[cp1251]{inputenc}
\usepackage[T2A]{fontenc}
\usepackage[english,russian]{babel}
\usepackage{amsfonts,amsmath,amsxtra,amsthm,amssymb,latexsym,mathrsfs,graphicx,setspace}
\usepackage{multirow}

\usepackage[ruled]{algorithm2e}

\SetAlCapHSkip{0pt}
\IncMargin{-\parindent}

\theoremstyle{plain} 
\newtheorem{theorem}{Теорема}
\newtheorem{lemma}{Лемма}

\newtheorem{corollary}{Следствие}

\theoremstyle{definition} 
\newtheorem{remark}{Замечание}

\numberwithin{equation}{section} \numberwithin{theorem}{section}
\numberwithin{lemma}{section} \numberwithin{proposition}{section}
\numberwithin{corollary}{section} \numberwithin{remark}{section}
\numberwithin{definition}{section} \numberwithin{example}{section}

\newcommand*{\affaddr}[1]{#1} 

\begin{document}

\renewcommand{\thesection}{}
\renewcommand{\thesubsection}{\arabic{subsection}}
\renewcommand{\thetheorem}{\arabic{theorem}}
\renewcommand{\theequation}{\arabic{equation}}
\renewcommand{\thedefinition}{\arabic{definition}}
\renewcommand{\thecorollary}{\arabic{corollary}}
\renewcommand{\thelemma}{\arabic{lemma}}
\renewcommand{\theexample}{\arabic{example}}
\renewcommand{\theremark}{\arabic{remark}}

\title{\textbf{
Адаптивный аналог метода Ю.~Е.~Нестерова для вариационных неравенств с сильно монотонным оператором\\
Some adaptive analog of Yu.~E.~Nesterov's method for variational inequalities with a strongly monotone operator
}}
\author{
\emph{Федор~Сергеевич~Стонякин}\\\emph{Fedor~Sergeevich~Stonyakin}\\\\\small
\affaddr{\textbf{Крымский федеральный университет имени В.~И.~Вернадского},\\\small Проспект Академика Вернадского, 4, Симферополь, 295007;}\\
\affaddr{\textbf{V.I. Vernadsky Crimean Federal University},\\\small 4 V. Vernadsky ave., Simferopol, 295007.}\\
\small e-mail: fedyor@mail.ru
}
\maketitle

\selectlanguage{russian}

Вариационные неравенства (ВН) нередко возникают в самых разных проблемах оптимизации и имеют многочисленные приложения в математической экономике, математическом моделировании транспортных потоков, теории игр и других разделах математики (см., например \cite{FaccPang_2003, Antipin_2017}).

Наиболее известным аналогом градиентного метода для ВН является экстраградиентный метод Г. М. Корпелевич \cite{Korpelevich}. Одним из современных вариантов экстраградиентного метода является проксимальный зеркальный метод А. С. Немировского \cite{Nemirovski_2004}. Недавно Ю. Е. Нестеровым в \cite{Nesterov_2015} предложен новый адаптивный метод решения задач выпуклой минимизации, который в случае липшицевости градиента целевой функции не требует знания никакой верхней оценки $\widehat{L} \geqslant L$ для этой константы Липшица $L$. На базе идеологии \cite{Nesterov_2015} в (\cite{Gasn_2017}, замечание 5.1 из раздела 5; см. также \cite{Gasn}) предложен похожий адаптивный аналог проксимального зеркального метода А. С. Немировского для вариационных неравенств с оператором, удовлетворяющим условию Липшица.

Хорошо известно, что в задачах оптимизации замена условия выпуклости функционала на сильную выпуклость приводит к существенно лучшей скорости сходимости методов. Аналогичный эффект имеет место и для вариационных неравенств, если оператор обладает свойством сильной монотонности. В \cite{Nesterov_2007, Nesterov_Doct} был предложен метод для вариационных неравенств с сильно монотонным липшицевым оператором. Этот метод представляет собой комбинацию двойственного экстраполяционного метода \cite{Nesterov_2005} и методики оценочных функций (см. раздел 2.2 из \cite{Nesterov_2010}).

В предлагаемой работе мы, в некоторой степени отталкиваясь от идеологии работы \cite{Gasn}, предложим адаптивный аналог метода Ю.~Е.~Нестерова для вариационных неравенств с липшицевым и сильно монотонным оператором, реализация которого не требует знания никакой верхней оценки $\hat{L} \geqslant L$ константы Липшица $L$ оператора $g$.

Будем рассматривать задачу нахождения решения $x^*=x^*(Q)$ вариационного неравенства
\begin{equation}\label{eq1}
\langle g(x^*),x^*-y\rangle\leqslant0\quad \forall y\in Q,
\end{equation}
где $g:Q\rightarrow\mathbb{R}^n$~--- сильно монотонный оператор с параметром $\mu>0$:
\begin{equation}\label{eq2}
\langle g(x)-g(y),x-y\rangle\geqslant\mu||x-y||^2\quad\forall x,y\in Q,
\end{equation}
$Q$~--- выпуклое замкнутое подмножество $\mathbb{R}^n$, $\langle\cdot,\cdot\rangle$~--- скалярное произведение в $\mathbb{R}^n$,
\begin{equation}\label{eq_norm}
\|x\| = \langle Bx, x\rangle^{1/2}
\end{equation}
есть некоторая евклидова норма в $\mathbb{R}^n$, где $B: \mathbb{R}^n \rightarrow \mathbb{R}^n$ --- фиксированный оператор $B = B^{T} > 0$. Будем полагать, что оператор $g$ удовлетворяет условию Липшица:
\begin{equation}\label{eq3}
||g(x)-g(y)||_*\leqslant L||x-y||\;\forall x,y\in Q
\end{equation}
для некоторой константы $L>0$, $||s||_* = \langle s, B^{-1}s\rangle^{1/2}$.

Напомним некоторые вспомогательные оценки, понятия и результаты из п. 3.2 диссертации Ю.~Е.~Нестерова \cite{Nesterov_Doct}. Отметим, что сильная монотонность $g$ означает, что для решения $x^*$ верны оценки при произвольном $y\in Q$:
\begin{equation}\label{eq4}
\langle g(y),x^*-y\rangle+\frac{\mu}{2}||y-x^*||^2\leqslant\langle g(x^*),x^*-y\rangle-\frac{\mu}{2}||y-x^*||^2\leqslant0.
\end{equation}

Неравенства \eqref{eq4} приводят к идее рассматривать следующую меру близости для оценки качества найденного приближённого решения $x$ ВН \eqref{eq1}:
\begin{equation}\label{eq5}
f(x)=\sup_{y\in Q}\left\{\langle g(y),x-y\rangle+\frac{\mu}{2}||y-x||^2\right\}.
\end{equation}

Отметим основные свойства $f$ из \eqref{eq5}.

\begin{theorem}\label{th1}
(Ю.~Е.~Нестеров, \cite{Nesterov_Doct}) Функция $f$ из \eqref{eq5} определена и сильно выпукла на $Q$ с параметром $\mu$. Более того, для всякого $x\in Q$ $f(x) \geqslant 0$ и $f(x)=0\Leftrightarrow x=x^*$.
\end{theorem}

Пусть в ходе работы некоторого алгоритма образовалась последовательность $\{y_i\}_{i=0}^{N}\subset Q$ и $\{\lambda_i\}_{i=0}^{N}$~--- некоторый набор положительных чисел. Тогда обозначим
\begin{equation}\label{eq6}
S_N=\sum\limits_{i=0}^{N}\lambda_i\text{ и }\widetilde{y}_N:=\frac{1}{S_N}\sum_{i=0}^{N}\lambda_i y_i \text{ --- усредненный выход работы алгоритма.}
\end{equation}

Неравенства \eqref{eq4} приводят к идее ввести следующую функцию зазора для оценки качества найденного решения:
\begin{equation}\label{eq7}
\Delta_N:=\max_{x\in Q}\left\{\sum\limits_{i=0}^{N}\lambda_i\left[\langle g(y_i),y_i-x\rangle-\frac{\mu}{2}||x-y_i||^2\right]\right\}.
\end{equation}

\begin{lemma}\label{lemma}
(Ю.~Е.~Нестеров, \cite{Nesterov_Doct}) Справедливо неравенство: $f(\widetilde{y}_N)\leqslant\frac{\Delta_N}{S_N}$.
\end{lemma}

Вслед за \cite{Nesterov_Doct} обозначим
\begin{equation}\label{eq8}
\varphi_{y}^{\beta}(x):=\langle g(y),y-x\rangle-\frac{\beta}{2}||x-y||^2, \quad
\Phi_k(x):=\sum_{i=0}^{k}\lambda_i\varphi_{y_i}^{\mu}(x)
\end{equation}
для произвольного параметра $\beta>0,\;k=0,1,2,\ldots$, а также $x,y\in Q$. Ясно, что функция $\varphi_{y}^{\beta}$ сильно вогнута с параметром $\beta$, а $\Phi_k(x)$ сильно вогнута с параметром $\mu S_k$. Заметим, что при этом $(k=0,1,2,\ldots,N)$
\begin{equation}
\Delta_k=\max_{x\in Q}\Phi_k(x).
\end{equation}

Напомним метод Ю.~Е. Нестерова для ВН с липшицевым сильно монотонным оператором \cite{Nesterov_2007, Nesterov_Doct}.
Опишем $(k+1)$-ю итерацию этого метода $(k=0,1,2,\ldots)$.

\begin{algorithm}
\SetAlgoNoLine
\caption{Метод для ВН с сильно монотонным оператором}
\label{alg1}
  $x_k:=\mathrm{arg}\max\limits_{x\in Q}\Phi_k(x)$, $y_{k+1}:=\mathrm{arg}\max\limits_{x\in Q}\varphi_{x_k}^{L}(x)$, $\lambda_{k+1}:=\frac{\mu}{L}S_k$.\\
    {\bf Выход:} $\widetilde{y}_{k+1}:=\frac{1}{S_{k+1}}\sum\limits_{i=0}^{k+1}\lambda_i y_i$.

\end{algorithm}

Для приведённого выше метода (алгоритм \ref{alg1}) согласно теореме 3.2.3 из \cite{Nesterov_Doct} в случае липшицева оператора $g$  с константой $L$ для числа обусловленности $\gamma=\frac{L}{\mu}$ и произвольного натурального $k$ верны оценки:
\begin{equation}\label{equvv11}
\frac{\mu}{2}||\widetilde{y}_k-x^*||^2\leqslant f(\widetilde{y}_k)\leqslant\left[f(y_0)+\frac{\mu(\gamma^2-1)}{2}||y_0-x^*||^2\right] \exp \left(-\frac{k}{\gamma+1}\right)\leqslant
\end{equation}
\begin{equation}\label{equvv12}
\leqslant
f(y_0)\cdot\gamma^2\cdot \exp \left(-\frac{k}{\gamma+1}\right).
\end{equation}

\begin{remark}
В конце пункта 3.2.2 \cite{Nesterov_Doct} было проведено сравнение алгоритма \ref{alg1} со стандартным проекционным методом вида:
\begin{equation}\label{eq 1}
x_{0}=\overline{x}\in Q,
\end{equation}
\begin{equation}\label{eq 2}
x_{k+1}=\pi_{Q}(x_{k}-\lambda B^{-1}g(x_{k})),\quad k\geq 0,
\end{equation}
где $\pi_{Q}(x)$ --- евклидова проекция точки $x$ на множество $Q$. Как отмечено в (\cite{Nesterov_Doct}, конец п. 3.2.2), у этого метода может быть медленная сходимость. В частности, при выборе оптимального шага $\displaystyle \lambda=\frac{\mu}{L^{2}}$ верна оценка:
\begin{equation}\label{Nest_est}
\|x_{k+1}-x^{*}\|^{2}\leq \|x_{k}-x^{*}\|^{2}\cdot\exp\left\{-\frac{k}{\gamma^{2}}\right\}.
\end{equation}
При больших значениях числа обусловленности $\gamma=\frac{L}{\mu}$ эта оценка может быть значительно хуже, чем \eqref{equvv11} -- \eqref{equvv12}. Также известно, что скорость сходимости \eqref{equvv11} -- \eqref{equvv12} не может быть улучшена никаким черноящичным методом, применяемым к задаче \eqref{eq1} -- \eqref{eq2} (см. замечание в конце п. 3.2.2 \cite{Nesterov_Doct}). В то же время с точки зрения сложности реализации метод Ю. Е. Нестерова не будет значительно сложнее метода \eqref{eq 1} -- \eqref{eq 2}: на каждой итерации требуется вычислить две проекции на множество и два значения оператора вместо одной проекции и одного значения в методе \eqref{eq 1} -- \eqref{eq 2}.
\end{remark}

Теперь перейдём к основным результатам работы и предложим адаптивный аналог метода Ю. Е. Нестерова для ВН \eqref{eq1} -- \eqref{eq2}. Положим изначально $\lambda_0:=1$, $y_0$~--- некоторое начальное приближение искомого решения и выберем некоторое $0<\beta_0\leqslant2L$, где $L$~--- константа Липшица для оператора $g$ из \eqref{eq3}.

\begin{remark}
Ввиду сильной монотонности оператора $g$ для произвольных различных $x$ и $y$ из множества $Q$ верно $g(x) \neq g(y)$. Поэтому выполнения условия $\beta_0 \leqslant 2 L$ можно добиться, выбрав
\begin{equation}\label{eqnullbeta}
\beta_0 := \frac{\|g(x) - g(y)\|_{*}}{\|x - y\|}
\end{equation}
для некоторых фиксированных различных $x$ и $y$ из $Q$.
\end{remark}

Опишем $(k+1)$-ю итерацию предлагаемого метода $(k=0,1,2,\ldots)$.
\newpage

\begin{algorithm}
\SetAlgoNoLine
\caption{Адаптивный метод для ВН с сильно монотонным оператором}
\label{alg2}
\begin{enumerate}
  \item $x_k:=\mathrm{arg}\max\limits_{x\in Q}\Phi_k(x), \;\beta_{k+1}:=\frac{\beta_k}{2}$.\\
  \item $y_{k+1}:=\mathrm{arg}\max\limits_{x\in Q}\varphi_{x_k}^{\beta_{k+1}}(x)$.\\
  \item {\bf Если} верно
    \begin{equation}\label{eq11}
    ||g(y_{k+1})-g(x_k)||_*\leqslant \sqrt{\beta_{k+1}(\beta_{k+1}+\mu)} \cdot ||y_{k+1}-x_k||,
    \end{equation}
    {\bf то} вычисляем $\lambda_{k+1}:=\frac{\mu}{\beta_{k+1}}S_k$, увеличиваем k на 1 и переходим к следующей итерации (п. 1).\\
    {\bf Иначе} $\beta_{k+1}:=2\cdot\beta_{k+1}$ и переходим к п. 2.\\
    {\bf Выход:} $\widetilde{y}_{k+1}:=\frac{1}{S_{k+1}}\sum\limits_{i=0}^{k+1}\lambda_i y_i$.
\end{enumerate}.
\end{algorithm}

\begin{remark}
При $\beta_{k+1} \geqslant L$ критерий выхода из итерации \eqref{eq11} заведомо выполнен, т.к.
$\sqrt{L(L+\mu)} > L $ при всяком $\mu > 0$. Поэтому после завершения итерации алгоритма \ref{alg2}
заведомо верно неравенство:
\begin{equation}\label{eq111}
    \beta_{k+1} < 2 L.
\end{equation}
Таким образом, константа $\beta_{k+1}$ не может неограниченно увеличится и максимальное её значение будет сопоставимо с $L$.
\end{remark}
\begin{remark}
Аналогично рассуждениям из (\cite{Nesterov_2015}, стр. 391) оценим количество операций п. 2 алгоритма \ref{alg2}. Пусть на $(k+1)$-й итераций их было $i_{k+1}$. Тогда ввиду деления на 2 в п. 1 алгоритма \ref{alg2} мы имеем:
\begin{equation}\label{equiv_iter_1}
\beta_{k+1} = \frac{1}{2} 2^{i_{k+1}-1} \beta_{k} = 2^{i_{k+1} - 2} \beta_{k},
\end{equation}
откуда
\begin{equation}\label{equiv_iter_2}
\sum\limits_{k = 0}^{N-1}i_{k+1} = 2 N + \sum\limits_{k = 0}^{N-1}\log_{2} \frac{\beta_{k+1}}{\beta_{k}} < 2 N + \log_{2} (2L) - \log_{2} (\beta_0).
\end{equation}

Таким образом, что за счёт повторения вычислений в п. 2 сложность работы предлагаемого алгоритма \ref{alg2} по сравнению с алгоритмом \ref{alg1} может увеличится не более, чем в 2 раза с точностью до постоянного слагаемого, зависящего от $\beta_0$ и $L$. Это означает, что трудоёмкость предлагаемого метода вполне сопоставима с трудоёмкостью исходного алгоритма \ref{alg1}. Однако при этом не требуется знания никакой константы $\widehat{L} \geqslant L$. Преимуществом также является возможное существенное увеличение скорости сходимости метода в конкретных задачах (см., например, таблицу \ref{tab1} ниже).
\end{remark}

Справедлива следующая
\begin{theorem}\label{th2}
При выполнении алгоритма \ref{alg2} для величин $\Delta_{k}$ из \eqref{eq7} верно неравенство $\Delta_{k+1}\leqslant\Delta_k$ для всякого целого неотрицательного $k$.
\end{theorem}
\begin{proof}
Ясно, что $\Phi_{k+1}(x)=\Phi_k(x)+\lambda_{k+1}\varphi_{y_{k+1}}^{\mu}(x)$. Тогда
$$
\Delta_{k+1}=\max_{x\in Q}\left\{\Phi_k(x)+\lambda_{k+1}\varphi_{y_{k+1}}^{\mu}(x)\right\}\leqslant
$$
$$
\leqslant\Delta_k+\max_{x\in Q}\left\{\langle\nabla\Phi_k(x_k),x-x_k\rangle-\frac{1}{2}\mu S_k||x-x_k||^2+\lambda_{k+1}\varphi_{y_{k+1}}^{\mu}(x)\right\}\leqslant
$$
$$
\leqslant\Delta_k+\max_{x\in Q}\left\{-\frac{1}{2}\mu S_k||x-x_k||^2+\lambda_{k+1}\left[\langle g(y_{k+1}),y_{k+1}-x\rangle-\frac{1}{2}\mu||x-y_{k+1}||^2\right]\right\}.
$$
В силу выбора $y_{k+1}$ из п. 2 алгоритма \ref{alg2} для всякого $x\in Q$ имеем:
$$
\langle-g(x_k)-\beta_{k+1} B(y_{k+1}-x_k),x-y_{k+1}\rangle\leqslant 0.
$$
Далее, с учётом равенства
\begin{equation}2\cdot \langle B(y_{k+1}-x_k),x-y_{k+1}\rangle = ||x-x_k||^2-||y_{k+1}-x_k||^2-||x-y_{k+1}||^2
\end{equation}
мы получаем оценки
$$
\langle g(y_{k+1}),y_{k+1}-x\rangle-\frac{\mu}{2}||x-y_{k+1}||^2=
$$
$$
=\langle g(y_{k+1})-g(x_k),y_{k+1}-x\rangle-\frac{\mu}{2}||x-y_{k+1}||^2+\langle g(x_k),y_{k+1}-x\rangle\leqslant
$$
$$
\leqslant ||g(y_{k+1})-g(x_k)||_*\cdot||y_{k+1}-x||-\frac{\mu}{2}||x-y_{k+1}||^2+\beta_{k+1}\langle B(y_{k+1}-x_k),x-y_{k+1}\rangle =
$$
$$
= ||g(y_{k+1})-g(x_k)||_*\cdot||y_{k+1}-x||-\frac{\mu}{2}||x-y_{k+1}||^2+ \frac{\beta_{k+1}}{2}||x-x_k||^2-\frac{\beta_{k+1}}{2}||y_{k+1}-x_k||^2-
$$
$$-\frac{\beta_{k+1}}{2}||x-y_{k+1}||^2 = ||g(y_{k+1})-g(x_k)||_*\cdot||y_{k+1}-x||-\frac{\beta_{k+1} + \mu}{2}||x-y_{k+1}||^2+
$$
$$
+ \frac{\beta_{k+1}}{2}||x-x_k||^2-\frac{\beta_{k+1}}{2}||y_{k+1}-x_k||^2
\leqslant - \frac{1}{2\mu} \left(||g(y_{k+1})-g(x_k)||_* - ||y_{k+1}-x)||\right)^2 +
$$
$$
+ \frac{||g(y_{k+1})-g(x_k)||^2_*}{2\cdot(\beta_{k+1} + \mu)} + \frac{\beta_{k+1}}{2}\left(||x-x_k||^2-||y_{k+1}-x_k||^2\right) \leqslant
$$
$$
\leqslant \frac{||g(y_{k+1})-g(x_k)||^2_*}{2\cdot(\beta_{k+1} + \mu)} + \frac{\beta_{k+1}}{2}\left(||x-x_k||^2-||y_{k+1}-x_k||^2\right).
$$
Поэтому ввиду \eqref{eq11}
$$
\langle g(y_{k+1}),y_{k+1}-x\rangle-\frac{\mu}{2}||x-y_{k+1}||^2\leqslant\frac{1}{2(\beta_{k+1} + \mu)}||g(y_{k+1})-g(x_k)||_*^2+
$$
$$
+\frac{\beta_{k+1}}{2}||x-x_k||^2 - \frac{\beta_{k+1}}{2}||y_{k+1}-x_k||^2 \leqslant \frac{\beta_{k+1}}{2}||x-x_k||^2,
$$
откуда с учетом $\mu S_k=\lambda_{k+1}\beta_{k+1}$ получаем требуемое.
\end{proof}

\begin{corollary}\label{cor1}
При выполнении алгоритма \ref{alg2} верно неравенство $f(\widetilde{y}_k)\leqslant\Delta_0 \exp \left(-\frac{k\mu}{\mu+\hat{\beta}}\right)$ для всякого натурального $k$, где $\hat{\beta}$ определяется следующим образом:
\begin{equation}\label{mean_beta}
1-\frac{\mu}{\mu+\hat{\beta}} = \sqrt[k]{\left(1-\frac{\mu}{\mu+\beta_1}\right)\left(1-\frac{\mu}{\mu+\beta_2}\right)\ldots\left(1-\frac{\mu}{\mu+\beta_k}\right)}.
\end{equation}
\end{corollary}
\begin{proof}
Действительно, $S_0=\lambda_0=1$ и для всякого $k=0,1,\ldots$ верно
$$
S_{k+1}=S_k+\lambda_{k+1}=\left(1+\frac{\mu}{\beta_{k+1}}\right)S_k.
$$
Далее, по лемме \ref{lemma}
$$
f(\widetilde{y}_k)\leqslant\frac{\Delta_k}{S_k}\leqslant\Delta_0\cdot\frac{S_0}{S_1}\cdot\frac{S_1}{S_2}\ldots\cdot\frac{S_{k-1}}{S_k}=
$$
$$
=\Delta_0\left(1-\frac{\mu}{\mu+\beta_1}\right)\left(1-\frac{\mu}{\mu+\beta_2}\right)\ldots\left(1-\frac{\mu}{\mu+\beta_k}\right) =
$$
$$
= \Delta_0\left(1-\frac{\mu}{\mu+\widehat{\beta}}\right)^k = \Delta_0\left(1-\frac{1}{1+\frac{\hat{\beta}}{\mu}}\right)^k\leqslant\Delta_0 \exp \left(-\frac{k\mu}{\mu+ \hat{\beta}} \right),
$$
что и требовалось.
\end{proof}

Из теорем \ref{th1} и \ref{th2}, а также следствия \ref{cor1} вытекает следующий результат, аналогичный теореме 3.2.3 из \cite{Nesterov_Doct}.

\begin{theorem}
Пусть оператор $g$ липшицев с константой $L > 0$ и сильно монотонен с параметром $\mu > 0$. Тогда при выполнении алгоритма \ref{alg2} для $\gamma=\frac{L}{\mu}$ и всякого натурального $k$ верны оценки:
\begin{equation}\label{equv11}
\frac{\mu}{2}||\widetilde{y}_k-x^*||^2\leqslant f(\widetilde{y}_k)\leqslant\left[f(y_0)+\frac{\mu(\gamma^2-1)}{2}||y_0-x^*||^2\right] \exp \left(-\frac{k}{1+\frac{\hat{\beta}}{\mu}}\right)\leqslant
\end{equation}
\begin{equation}\label{equv12}
\leqslant
f(y_0)\cdot\gamma^2\cdot \exp \left(-\frac{k}{1+\frac{\hat{\beta}}{\mu}}\right).
\end{equation}
\end{theorem}

Отметим, что оценки \eqref{equv11} -- \eqref{equv12} могут оказаться лучше \eqref{equvv11} -- \eqref{equvv12} из \cite{Nesterov_Doct}, поскольку $\frac{\hat{\beta}}{\mu}$ может оказаться меньше $\gamma$. Далее это наглядно продемонстрировано на примере численного эксперимента для задачи \eqref{Problem}.

\begin{remark}
Рассмотрим модификацию алгоритма \ref{alg2}, которая исключает уменьшение константы $\beta_{k+1}$ в ходе работы метода. Это
даёт возможность сделать вывод о несущественном увеличении трудоёмкости по сравнению с методом Ю. Е. Нестерова (алгоритм \ref{alg1}). Изначально положим $\lambda_0:=1$, $y_0$~--- некоторое начальное приближение искомого решения и выберем некоторое $0<\beta_0\leqslant2L$ (см. \eqref{eqnullbeta}), где $L$ --- константа Липшица для оператора $g$ из \eqref{eq3}.  Опишем $(k+1)$-ю итерацию предлагаемой модификации алгоритма \ref{alg2} $(k=0,1,2,\ldots)$.
\begin{algorithm}
\SetAlgoNoLine
\caption{Модификация алгоритма \ref{alg2}}
\label{alg3}
\begin{enumerate}
  \item $x_k:=\mathrm{arg}\max\limits_{x\in Q}\Phi_k(x), \;\beta_{k+1}:=\beta_k$.\\
  \item $y_{k+1}:=\mathrm{arg}\max\limits_{x\in Q}\varphi_{x_k}^{\beta_{k+1}}(x)$.\\
  \item {\bf Если} верно
    \begin{equation}\label{eq12}
    ||g(y_{k+1})-g(x_k)||_*\leqslant \sqrt{\beta_{k+1}(\beta_{k+1}+\mu)} \cdot ||y_{k+1}-x_k||,
    \end{equation}
    {\bf то} вычисляем $\lambda_{k+1}:=\frac{\mu}{\beta_{k+1}}S_k$, увеличиваем $k$ на 1 и переходим к следующей итерации (п. 1).\\
    {\bf Иначе} $\beta_{k+1}:=2\cdot\beta_{k+1}$ и переходим к п. 2.\\
    {\bf Выход:} $\widetilde{y}_{k+1}:=\frac{1}{S_{k+1}}\sum\limits_{i=0}^{k+1}\lambda_i y_i$.
\end{enumerate}.
\end{algorithm}
\newpage
Поскольку $\beta_{k+1}$ может лишь увеличиваться, то по сравнению с алгоритмом \ref{alg1} количество вычислений согласно п. 2 возрастёт лишь не более, чем на
\begin{equation}
\left\lceil\log_{2} \frac{2L}{\beta_0}\right\rceil.
\end{equation}
\end{remark}

Для демонстрации преимуществ алгоритмов \ref{alg2} и \ref{alg3} по сравнению с алгоритмом \ref{alg1} были проведены вычислительные эксперименты для вариационного неравенства с оператором $\displaystyle g:\mathbb{R}^{20}\rightarrow \mathbb{R}^{20}$ вида $$g(x_{1},x_{2},\ldots,x_{20})=$$
\begin{equation}\label{Problem}
= \left(\exp\left(x_{1}+\frac{x_{2}}{10 \exp(3)}\right), \exp\left(x_{2}+\frac{x_{3}}{10 \exp (3)}\right),\ldots, \exp\left(x_{20}+\frac{x_{1}}{10 \exp(3)} \right)\right).
\end{equation}
В качестве множества $Q$ выберем единичный шар с центром в нуле
$$
Q=\left\{x=(x_{1},x_{2},\ldots,x_{20})\;|\; x_{1}^{2}+x_{2}^{2}+\ldots+x_{20}^{2}\leqslant 1\right\}
$$
для стандартной евклидовой нормы в $\mathbb{R}^{20}$ (т.е. здесь оператор $B$ в \eqref{eq_norm} мы полагаем тождественным):
$$\|x\|:=\sqrt{x_{1}^{2}+x_{2}^{2}+\ldots+x_{20}^{2}}.$$

Пусть $x=(x_{1},x_{2},\ldots,x_{20})$ и $y=(y_{1},y_{2},\ldots,y_{20})$ --- два вектора из $Q$. Очевидно, что оператор $g$ не является потенциальным:
$$\frac{\partial}{\partial x_{2}}\left(\exp\left(x_{1}+\frac{x_{2}}{10 \exp(3)}\right)\right)=\frac{1}{10 \exp(3)} \exp\left(x_{1}+\frac{x_{2}}{10 \exp(3)}\right),
$$
$$\frac{\partial}{\partial x_{1}}\left(\exp\left(x_{2}+\frac{x_{3}}{10 \exp(3)}\right)\right)=0.$$

Покажем, что оператор $g$ удовлетворяет условию Липшица и сильно монотонен на $Q$. По теореме о среднем для произвольных $i,j=\overline{1,20}$ имеем:
$$
\exp\left(x_{i}+\frac{x_{j}}{10 \exp(3)}\right) - \exp\left(y_{i}+\frac{y_{j}}{10 \exp(3)}\right) = \exp\left(x_{i}+\frac{x_{j}}{10 \exp(3)}\right) - \exp\left(y_{i}+\frac{x_{j}}{10 \exp(3)} \right) +
$$
\begin{equation}\label{eq examp2}
+\exp\left(y_{i}+\frac{x_{j}}{10 \exp(3)}\right) - \exp\left(y_{i}+\frac{y_{j}}{10 \exp(3)} \right) = \exp\left(\alpha_{i}+\frac{x_{j}}{10 \exp(3)} \right)(x_{i}-y_{i})+
\end{equation}
$$
+\frac{1}{10 \exp(3)} \exp\left(y_{i}+\frac{\gamma_{j}}{10 \exp(3)}\right)(x_{j}-y_{j})
$$
для некоторых $\alpha_{i}$ и $\gamma_{j}:\;|\alpha_{i}|\leqslant 1$ и
$|\gamma_{j}|\leqslant 1$ ($\alpha_{i}$ и $\gamma_{j}$ лежат между $x_{i},y_{i}$ и $x_{j},y_{j}$ соответственно).
Ясно, что
$$\left|\alpha_{i}+\frac{x_{j}}{10 \exp(3)}\right|\leqslant\sqrt{1+\left(\frac{1}{10 exp(3)}\right)^{2}}\sqrt{\alpha_{i}^{2}+x_{j}^{2}}<\sqrt{2},$$
а также $\displaystyle\left|y_{i}+\frac{\gamma_{j}}{10 \exp(3)}\right|<\sqrt{2}$. Поэтому \eqref{eq examp2} означает, что
$$\left|\exp \left(x_{i}+\frac{x_{j}}{10 \exp(3)} \right) - \exp \left(y_{i}+\frac{y_{j}}{10 \exp(3)}\right)\right|<exp(\sqrt{2})|x_{i}-y_{i}|+ $$
$$
+ \exp (\sqrt{2}-3)\frac{|x_{j}-y_{j}|}{10} < \exp (\sqrt{2})\left(|x_{i}-y_{i}|+\frac{|x_{j}-y_{j}|}{10}\right).
$$
Далее, с учётом неравенства $(a+b)^2 \leqslant 2(a^2 + b^2)$, имеем
$$\|g(x)-g(y)\|^{2}<2exp(2\sqrt{2})\left(1+\frac{1}{100}\right)\|x-y\|^{2}\quad\forall\;x,y\in Q,$$
$$\|g(x)-g(y)\|<\frac{\sqrt{202} \exp(\sqrt{2})}{10}\|x-y\| \quad\forall\;x,y\in Q,$$
т.е. оператор $g$ удовлетворяет свойству Липшица с константой $\displaystyle L=\frac{\sqrt{202}}{10}\exp(\sqrt{2})$.

Далее, \eqref{eq examp2} означает, что для произвольных $x,y\in Q$
$$\langle g(x)-g(y),x-y\rangle=\sum_{k=1}^{20}c_{k}(x_{k}-y_{k})^{2}+\sum_{k=1}^{19}d_{k}(x_{k}-y_{k})(x_{k+1}-y_{k+1})+d_{20}(x_{20}-y_{20})(x_{1}-y_{1}),$$
где
$$\displaystyle c_{k}> \exp(-\sqrt{2}),\;d_{k}<\frac{\exp(\sqrt{2}-3)}{10}<\frac{\exp(-\sqrt{2})}{10}$$
для всякого $k=\overline{1,20}$.
Учитывая неравенство $\displaystyle 2 ab\leqslant a^{2}+b^{2}$, получаем
$$\langle g(x)-g(y),x-y\rangle> \exp (-\sqrt{2})\sum_{k=1}^{20}(x_{k}-y_{k})^{2}-\frac{\exp (-\sqrt{2})}{10}\sum_{k=1}^{20}(x_{k}-y_{k})^{2}=\frac{9}{10} \exp(-\sqrt{2})\|x-y\|^{2},$$
т.е. оператор $g$ сильно монотонен с параметром $\displaystyle \mu=\frac{9}{10} \exp (-\sqrt{2})$.

\begin{table}
\centering
\caption{"Сравнение результатов работы алгоритмов \ref{alg1} и \ref{alg2}".}
\label{tab1}
\begin{tabular}{|c|c|c|c|c|c|c|}
\hline
$N$  & $exp\left(\frac{-k}{1+\frac{L}{\mu}}\right)$ &Время,\;мс& $exp\left(\frac{-k}{1+\frac{\widehat{\beta}}{\mu}}\right)$ &Время,\;мс& $\beta_N$ & $\hat{\beta}$\\ \hline
3   & 8.9742e-01                           & 1    & 3.3880e-01                                     & 2     & 2.1447e-01   & 3.8766e-01  \\ \hline
6   & 8.0536e-01                           & 1    & 2.0270e-02                                     & 3     & 2.6809e-02   & 2.6809e-02  \\ \hline
9   & 7.2274e-01                           & 1    & 4.9199e-04                                     & 4     & 3.3512e-03   & 3.9726e-02  \\ \hline
12  & 6.4860e-01                           & 2    & 1.2773e-05                                     & 5     & 4.1889e-04   & 1.4210e-02  \\ \hline
15  & 5.8207e-01                           & 2    & 4.3275e-07                                     & 5     & 5.2362e-05   & 5.1801e-03  \\ \hline
18  & 5.2236e-01                           & 2    & 1.7770e-08                                     & 7     & 6.5452e-06   & 1.8911e-03  \\ \hline
21  & 4.6878e-01                           & 3    & 8.0981e-10                                     & 8     & 8.1815e-07   & 6.8756e-04  \\ \hline
24  & 4.2069e-01                           & 3    & 3.8794e-11                                     & 8     & 1.0227e-07   & 2.4877e-04  \\ \hline
27  & 3.7753e-01                           & 4    & 1.9004e-12                                     & 9     & 1.2784e-08   & 8.9622e-05  \\ \hline
30  & 3.3881e-01                           & 4    & 9.3990e-14                                     & 9     & 1.5980e-09   & 3.2176e-05  \\ \hline
33  & 3.0405e-01                           & 5    & 4.6670e-15                                     & 10    & 1.9974e-10   & 1.1521e-05  \\ \hline
36  & 2.7286e-01                           & 5    & 2.3211e-16                                     & 10    & 2.4968e-11   & 4.1168e-06  \\ \hline
39  & 2.4487e-01                           & 6    & 1.1551e-17                                     & 12    & 3.1210e-12   & 1.4687e-06  \\ \hline
42  & 2.1975e-01                           & 7    & 5.7501e-19                                     & 13    & 3.9013e-13   & 5.2327e-07  \\ \hline
45  & 1.9721e-01                           & 8    & 2.8626e-20                                     & 14    & 4.8766e-14   & 1.8625e-07  \\ \hline
\end{tabular}
\end{table}

\begin{table}
\centering
\caption{"Сравнение результатов работы алгоритмов \ref{alg1} и \ref{alg3}".}
\label{tab2}
\begin{tabular}{|c|c|c|c|c|c|c|}
\hline
$N$  & $exp\left(\frac{-k}{1+\frac{L}{\mu}}\right)$ &Время,\;мс& $exp\left(\frac{-k}{1+\frac{\widehat{\beta}}{\mu}}\right)$ &Время,\;мс& $\beta_N$ & $\hat{\beta}$\\ \hline
3   & 8.9742e-01                           & 1    & 7.1227e-01                                     & 2     & 1.7158e+00   & 1.7158e+00  \\ \hline
6   & 8.0536e-01                           & 1    & 5.0732e-01                                     & 3     & 1.7158e+00   & 1.7158e+00  \\ \hline
9   & 7.2274e-01                           & 1    & 3.6135e-01                                     & 4     & 1.7158e+00   & 1.7158e+00  \\ \hline
12  & 6.4860e-01                           & 2    & 2.5738e-01                                     & 5     & 1.7158e+00   & 1.7158e+00  \\ \hline
15  & 5.8207e-01                           & 2    & 1.8332e-01                                     & 5     & 1.7158e+00   & 1.7158e+00  \\ \hline
18  & 5.2236e-01                           & 2    & 1.3057e-01                                     & 7     & 1.7158e+00   & 1.7158e+00  \\ \hline
21  & 4.6878e-01                           & 3    & 9.3003e-02                                     & 8     & 1.7158e+00   & 1.7158e+00  \\ \hline
24  & 4.2069e-01                           & 3    & 6.6243e-02                                     & 8     & 1.7158e+00   & 1.7158e+00  \\ \hline
27  & 3.7753e-01                           & 4    & 4.7183e-02                                     & 9     & 1.7158e+00   & 1.7158e+00  \\ \hline
30  & 3.3881e-01                           & 4    & 3.3607e-02                                     & 9     & 1.7158e+00   & 1.7158e+00  \\ \hline
33  & 3.0405e-01                           & 5    & 2.3937e-02                                     & 10    & 1.7158e+00   & 1.7158e+00  \\ \hline
36  & 2.7286e-01                           & 5    & 1.7049e-02                                     & 10    & 1.7158e+00   & 1.7158e+00  \\ \hline
39  & 2.4487e-01                           & 6    & 1.2144e-02                                     & 12    & 1.7158e+00   & 1.7158e+00  \\ \hline
42  & 2.1975e-01                           & 7    & 8.6496e-03                                     & 13    & 1.7158e+00   & 1.7158e+00  \\ \hline
45  & 1.9721e-01                           & 8    & 6.1608e-03                                     & 14    & 1.7158e+00   & 1.7158e+00  \\ \hline
\end{tabular}
\end{table}

Мы применили алгоритмы \ref{alg1}, \ref{alg2} и \ref{alg3} к вариационному неравенству для оператора $g$ из \eqref{Problem} с параметрами $L = \frac{\sqrt{202}}{10}\exp(\sqrt{2})$, $\mu=\frac{9}{10} \exp (-\sqrt{2})$, начального приближения $y_0 = (0.2, 0.2, ..., 0.2) \in Q$ и в соответствии с \eqref{eqnullbeta}
\begin{equation}\label{eqexampnullbeta}
\beta_0 = \frac{\|g(1, 0, 0, ..., 0) - g(0, 1, 0, ..., 0)\|}{\sqrt{2}}
\end{equation}
для стандартной евклидовой нормы в $\mathbb{R}^{20}$.

Результаты сравнения работы алгоритмов \ref{alg1} и \ref{alg2} (а также алгоритмов \ref{alg1} и \ref{alg3}) представлены в сравнительных таблицах \ref{tab1} и \ref{tab2}, где $N$ --- количество итераций работы этих алгоритмов, время работы алгоритмов указано в миллисекундах. Все вычисления были произведены с помощью CPython 3.6.4 на компьютере с 3-ядерным процессором AMD Athlon II X3 450 с тактовой частотой 803,5 МГц на каждое ядро. ОЗУ компьютера составляла 8 Гб.

Как видим из таблицы \ref{tab1}, скорость сходимости для предлагаемого нами алгоритма \ref{alg2} существенно выше скорости сходимости алгоритма \ref{alg1}. Это получается за счёт значительного уменьшения констант $\beta_N$ на итерациях в ходе работы алгоритма, а также предлагаемого нами их усреднения в \eqref{mean_beta}. Из таблицы \ref{tab2} видим, что скорость сходимости для предлагаемого нами алгоритма \ref{alg3} выше, чем для алгоритма \ref{alg1}, но уже не так существенно, как для алгоритма \ref{alg2}. При этом время работы алгоритма \ref{alg3} меньше, чем время работы алгоритма \ref{alg2}. По сути, преимущество алгоритма \ref{alg3} перед алгоритмом \ref{alg1} для рассматриваемого примера определяется, прежде всего, возможностью выбора начальной константы $\beta_0$ \eqref{eqexampnullbeta} согласно предлагаемому нами способу в замечании 2 без использования какой-либо оценки $\hat{L} \geqslant L$ константы Липшица $L$ оператора $g$.

\bigskip

{\bf Благодарности}. Автор выражает огромную благодарность Александру Владимировичу Гасникову и Юрию Евгеньевичу Нестерову, а также неизвестным рецензентам за полезные обсуждения и комментарии.

\end{document}